\documentclass[twoside, 12pt]{article}
\usepackage{mathrsfs}
\usepackage{amssymb, amsmath, mathrsfs, amsthm}
\usepackage{graphicx}
\usepackage{color}
\usepackage{tikz}
\usepackage[top=2cm, bottom=2cm, left=2cm, right=2cm]{geometry}
\usepackage{float, caption, subcaption}
\usepackage{diagbox}
\usepackage{algorithm}
\usepackage{algpseudocode}
\usepackage{amsmath}
\usepackage[colorlinks,
  linkcolor=blue,
 anchorcolor=blue,
 citecolor=blue]{hyperref} 
\input{epsf}

\newcommand{\bd}{\begin{description}}
\newcommand{\ed}{\end{description}}
\newcommand{\bi}{\begin{itemize}}
\newcommand{\ei}{\end{itemize}}
\newcommand{\be}{\begin{enumerate}}
\newcommand{\ee}{\end{enumerate}}
\newcommand{\beq}{\begin{equation}}
\newcommand{\eeq}{\end{equation}}
\newcommand{\beqs}{\begin{eqnarray*}}
\newcommand{\eeqs}{\end{eqnarray*}}

\definecolor{DarkGreen}{rgb}{0.2, 0.6, 0.3}

\usepackage{multicol}
\usepackage{amsthm}
\usepackage{multirow}
\usepackage{makecell}

\newtheorem{theorem}{Theorem}[section]

\newtheorem{lemma}{Lemma}[section]
\newtheorem{definition}{Definition}[section]
\newtheorem{corollary}[theorem]{Corollary}
\newtheorem{case}{Case}
\newtheorem{subcase}{Subcase}[case]
\newtheorem{claim}{Claim}

\newtheorem{fact}{Fact}
\newtheorem{proposition}{Proposition}[section]

\newtheorem{problem}{Problem}

\providecommand{\prec}{\preceq}               

\begin{document}
\title{\textbf{On the size edge-ordered Ramsey numbers of graphs}
\footnote{Supported by the National Science Foundation of China
(Nos. 12471329 and 12061059).}
}
\author{Yanyan Song\footnote{School of Mathematics and Statistis, Qinghai Normal University, Xining, Qinghai 810008, China. {\tt songyanyan@hrbeu.edu.cn}}, \ \ Yaping Mao \footnote{Academy of Plateau
Science and Sustainability, and School of Mathematics and Statistics, Qinghai Normal University, Xining, Qinghai 810008, China. {\tt yapingmao@outlook.com; myp@qhnu.edu.cn}}
}

\date{}
\maketitle
\begin{abstract}
For edge-ordered graphs $G^{\prec}$ and $H^{\prec}$, the size edge-ordered Ramsey number $\hat{r}_{\text{edge}}(G^{\prec}, H^{\prec})$ is defined as the smallest integer $m$ for which there exists an edge-ordered graph $F^{\prec}$ (with underlying graph $F$) having $m$ edges, such that every $2$-coloring of the edges of $F^{\prec}$ contains a monochromatic edge-ordered subgraph isomorphic to $G^{\prec}$ or a monochromatic edge-ordered subgraph isomorphic to $H^{\prec}$.
Fox and Li posed a foundational question: which families of edge-ordered graphs have linear or near-linear size edge-ordered Ramsey numbers?
In this paper, we apply Szemerédi's regularity lemma to prove that, even for sparse graph families, specifically the well-defined class of edge-ordered book graphs, the size edge-ordered Ramsey numbers of this family exhibit non-linear growth.
Furthermore, we show that three families of edge-ordered graphs exhibit linear or near-linear size edge-ordered Ramsey numbers.\\[2mm]
\end{abstract}

\section{Introduction}
\indent
\par
Throughout this paper, all graphs are presumed to be undirected, finite, and simple, any undefined terms or notations conform to the established conventions in \cite{BM08}. Let $G = (V(G), E(G))$ denote a graph comprising vertex set $V(G)$ and edge set $E(G)$. We employ $K_n$, $P_n$, and $C_n$ to denote the complete graph, path, and cycle (in that order) on $n$ vertices, with $nK_2$ representing an $n$-edge matching. The \emph{book graph} $B_{m,n}$ is defined as the graph formed by taking $n$ copies of $K_{m+1}$ and identifying them along a common complete subgraph $K_m$. We refer to this $K_m$ as the \textit{spine}, while the vertices belonging to the independent set are termed the \textit{pages}.

Ramsey theory explores the unavoidable emergence of organized substructures within large colored structures, where Ramsey numbers act as a core tool for capturing this inherent property \cite{GRS90}. The fundamental Ramsey number identifies the minimal graph size ensuring the existence of a monochromatic subgraph.

\begin{definition}
The \textit{Ramsey number} $r(G,H)$ is the smallest $N$ such that any $2$-coloring of the edges of a complete graph on $N$ vertices, there exists a red copy of $G$ or a blue copy of $H$.
\end{definition}

Classical Ramsey numbers focus on unordered graphs, but recent work has extended this framework to edge-ordered graphs where edges carry a total order-introducing new questions about how edge order impacts Ramsey-type phenomena.

An \textit{edge-ordered graph} $G^{\prec}$ is a pair $(G, \prec)$ consisting of a graph $G = (V(G), E(G))$ and a linear ordering $\prec$ defined on the edge set $E(G)$. An edge-ordered graph $G^{\prec}$  as an \textit{edge-ordered subgraph} of an edge-ordered graph $H^{\prec}$ if $G$ is a subgraph of $H$ and the edge-ordering of $G$ constitutes a suborder of $H$’s edge-ordering. 
Two edge-ordered graphs $G^{\prec_1}$ and $H^{\prec_2}$ are \emph{isomorphic} if there exists a graph isomorphism 
$\varphi \colon G \to H$ whose induced edge-map $\varphi_E \colon E(G) \to E(H)$ satisfies
$e_1 \prec_1 e_2$ if and only if $\varphi_E(e_1) \prec_2 \varphi_E(e_2)$
for all $e_1,e_2 \in E(G)$.
For a vertex $v \in V(G^{\prec})$, $d_{G^{\prec}}(v)$ stands for the degree of $v$ in $G$, and $\delta(G^{\prec})$ and $\Delta(G^{\prec})$ denote the minimum degree and maximum degree of $G^{\prec}$, respectively. The graph $G^{\prec}-H^{\prec}$ is obtained by eliminating all vertices of $H^{\prec}$ from $G^{\prec}$. If $H^{\prec}=\{v\}$, we simply write $G^{\prec}-v$. For any two edge-ordered graphs $G^{\prec}$ and $H^{\prec}$, their union $G^{\prec} \cup H^{\prec}$ possesses vertex set $V(G^{\prec}) \cup V(H^{\prec})$ and edge set $E(G^{\prec}) \cup E(H^{\prec})$.

Balko and Vizer \cite{BV20} extended Ramsey numbers to this setting, defining edge-ordered Ramsey numbers to capture forced monochromatic edge-ordered subgraphs.

\begin{definition}
The \emph{edge-ordered Ramsey number} $r_{\text {edge }}(G^{\prec},H^{\prec})$, is defined as the smallest integer $N$ such that there exists an edge-ordered complete graph $K_N^{\prec}$ of $K_N$ such that any $2$-coloring of the edges of $K_N^{\prec}$ contains a red edge-ordered copy of $G^{\prec}$ or a blue edge-ordered copy of $H^{\prec}$.
\end{definition}

A trivial but critical observation is as follows:
\begin{equation}\label{e-1}
 r(G,H) \leq r_{\text{edge}}(G^{\prec},H^{\prec})
 \end{equation}
for any edge-ordered graphs $G^{\prec}$ and $H^{\prec}$, since the edge-ordered Ramsey number must account for both the underlying graph structure and its edge order. For edge-ordered complete graphs $K_m^{\prec}, K_n^{\prec}$ with lexicographic edge order, this inequality becomes an equality: $r_{\text{edge}}(K_m^{\prec}, K_n^{\prec}) = r(K_m, K_n)$. This is because the lexicographic order of $K_{r(K_m,K_n)}^{\prec}$ forces a monochromatic edge-ordered complete subgraph whenever it forces a monochromatic $K_m^{\prec}$ or a monochromatic $K_n^{\prec}$.

To the best of our knowledge, Ramsey numbers associated with edge-ordered graphs had not been explored in the academic literature prior to \cite{BV19}. In contrast, Ramsey numbers for graphs equipped with ordered vertex sets have undergone substantial investigation in recent years (see \cite{BCKK15, BCKK20, BJV19, BP23, BV22, CP02, CFLS17, Ro19}).

However, the general existence of edge-ordered Ramsey numbers is not immediately evident. Balko and Vizer \cite{BV20} proved that such numbers do indeed exist; they also showed that $r_{\text{edge}}(G^{\prec}) \leq 2^{O(n^3 \log n)}$ holds for any bipartite edge-ordered graph $G^{\prec}$ on $n$ vertices. Furthermore, they defined a natural class of edge orderings, termed lexicographic edge-orderings, for which we can establish significantly sharper upper bounds on the corresponding edge-ordered Ramsey numbers.

The \emph{edge-ordered Ramsey number} $r_{\text{edge}}(G^{\prec}; q)$ is defined as the smallest integer $N$ such that there exists an edge-ordered complete graph $K_N^{\prec}$ where any $q$-coloring of the edges of $K_N^{\prec}$ contains a monochromatic edge-ordered copy of $G^{\prec}$.
Fox and Li \cite{FL20} proved that there exists a constant $c$ such that $r_{\text{edge}}(H^{\prec}; q) \leq 2^{c q^n n^{2q-2} \log^q n}$ for any edge-ordered graph $H^{\prec}$ with $n$ vertices. Moreover, they established a polynomial bound on the edge-ordered Ramsey number of graphs with bounded degeneracy. Lastly, they extended this result to graphs where each edge carries a label (with no inherent ordering imposed on the labels).

Erd\H{o}s et al. \cite{EFRS78} defined the size-Ramsey number as follows.

\begin{definition}\label{def-er}
For graphs $G$ and $H$, the size-Ramsey number $\hat{r}(G,H)$ is the smallest number $m$ such that there exists a graph $F$ on $m$ edges such that any $2$-coloring of the edges of $F$, there exists either a red copy of $G$ or a blue copy of $H$.    
\end{definition}

If $G=H$, then we write $\hat{r}(G)$ for short, instead of $\hat{r}(G,G)$. 

The concept of size Ramsey numbers enables more precise investigation of the minimality of the target graph $F$. One can always select a sufficiently large complete graph to serve as the target graph $F$, so $\binom{r(G,H)}{2}$ forms a trivial upper bound for $\hat{r}(G,H)$. This bound is also tight when $G,H$ are both complete graphs \cite{EFRS78}, but for other graphs $G,H$, the optimal target graph is typically far sparser. For additional recent findings regarding the size-Ramsey numbers of graphs, see \cite{Be83, BKMMMMP21, CJKMMRR19, CFW23, HJKMR20, HKTL95, KLWY21, KRSS11, RS00, Ti24}.

Fox and Li \cite{FL20} extended this to edge-ordered graphs, defining the size edge-ordered Ramsey number to capture the minimum number of edges needed to force a monochromatic edge-ordered subgraph.

\begin{definition}\label{def-seorn}
For edge-ordered graphs $G^{\prec}$ and $H^{\prec}$, the \textit{size edge-ordered Ramsey number} $\hat{r}_{\text {edge }}(G^{\prec},H^{\prec})$ is defined as the minimum number $m$ such that there exists an edge-ordered graph $F^{\prec}$ of $F$ with $m$ edges such that any $2$-coloring of the edges of $F^{\prec}$ has a red edge-ordered copy of $G^{\prec}$ or a blue edge-ordered copy of $H^{\prec}$.
\end{definition}

If $G^{\prec}=H^{\prec}$, then we write $\hat{r}_{\text {edge }}(G^{\prec})$ for short, instead of $\hat{r}_{\text {edge }}(G^{\prec},G^{\prec})$.

A trivial yet critical observation is that $\hat{r}(G,H) \leq \hat{r}_{\text{edge}}(G^{\prec},H^{\prec})$ holds for any edge-ordered graphs $G^{\prec}, H^{\prec}$. This is because the size edge-ordered Ramsey number needs to consider both the underlying graph structure and its edge order. Similarly, for edge-ordered graphs $G^{\prec}$ and $H^{\prec}$, we also have
\begin{equation}\label{e1}
\hat{r}(G,H)  \leq \hat{r}_{\text{edge}}(G^{\prec},H^{\prec}) \leq\binom{r_{\text{edge}}(G^{\prec},H^{\prec})}{2}.
\end{equation}

\begin{proposition}\label{pro-1}
If \( G^{\prec} \) is an edge-ordered subgraph of \( H^{\prec} \), then
\[
\hat{r}_{\text{edge}}(G^{\prec}) \leq \hat{r}_{\text{edge}}(H^{\prec}).
\]    
\end{proposition}

\begin{proof}
Let \( m = \hat{r}_{\text{edge}}(H^{\prec})\). By definition there exists an edge-ordered  graph \( F^{\prec} \) satisfying \( |E(F^{\prec})| = m \) and any $2$-coloring of the edges of $F^{\prec}$ contains a monochromatic edge-ordered copy of $H^{\prec}$. Since \( G^{\prec} \) is an edge-ordered subgraph of \( H^{\prec}\), it follows that there is a monochromatic edge-ordered  graph \( H^{\prec} \), then $\hat{r}_{\text{edge}}(G^{\prec}) \leq \hat{r}_{\text{edge}}(H^{\prec}).$
\end{proof}

We define the edge-ordered book graph $B_{m,n}^{\prec}$ by endowing the book graph $B_{m,n}$ with an edge order. Let the vertices of the spine be $u_1, \ldots, u_m$ and the vertices of the pages be $v_1, \ldots, v_n$. The edge order $\prec_{B_{m,n}}$ is specified as follows:
for any two spine edges $(u_i, u_j)$ and $(u_s, u_t)$ with $1 \leq i <j \leq m$ and $1 \leq s <t \leq n$, we have
$$
(u_i, u_j) \prec_{B_{m,n}} (u_s, u_t) \Longleftrightarrow (i < s) \text{ or } (i = s \text{ and } j < t),
$$
and for any two page-spine edges $(v_i, u_j)$ and $(v_s, u_t)$, we have
$$
(u_{m-1}, u_{m}) \prec_{B_{m,n}} (v_i, u_j) \prec_{B_{m,n}} (v_s, u_t) \Longleftrightarrow (i < s) \text{ or } (i = s \text{ and } j < t).
$$

Edge-ordered bipartite graph $H^{\prec}$: the underlying graph is a bipartite graph $H$ with two partite sets $A$ and $B$. Let $V(A) = \{u_1, \ldots, u_s\}$ and $V(B) = \{v_1, \ldots, v_t\}$. The edge order $\prec_H$ is defined as the vertex lexicographical order: for any two edges $(u_i, v_j)$ and $(u_k, v_l)$ with $1 \leq i <j \leq s$ and $1 \leq k <l \leq t$, we have 
$$
(u_i, v_j) \prec_H (u_k, v_l) \Longleftrightarrow (i < k) \text{ or } (i = k \text{ and } j < l).
$$
All subsequent bipartite graphs satisfy this edge order. In particular, we denote the edge-ordered complete bipartite graph as $K_{s,t}^{\prec}$.

A question was asked by 
Fox and Li \cite{FL20} for size edge-ordered Ramsey numbers.

\begin{problem}{\upshape \cite{FL20}}\label{pro2}
What families of edge-ordered graphs, have linear or near-linear size edge-ordered Ramsey number?    
\end{problem}

Motivated by Problem \ref{pro2}, we study the size edge-ordered Ramsey number of graphs. 
Our results are as follows.

\begin{itemize}
  \item For any $m \geq 25$ and every $n \geq 300m^2$, Conlon et al. \cite{CFW23} showed that 
$\hat{r}(B_{m,n}) \geq m 2^m n^2 / 1200.$ 
In Section 2, we use Szemer\'{e}di's regularity lemma to prove that $\hat{r}_{\text{edge}}(B_{m,n}^{\prec}) \leq (2m^2 + m 2^{m+2})n^2-mn.
$ From (\ref{e1}), we have $\hat{r}_{\text{edge}}(B_{m,n}^{\prec}) = \Theta(m 2^m n^2)$ for sufficiently large $n$ and a fixed integer $m$ (Theorem \ref{th-5-1}). 

\item Theorem \ref{th-5-1} and (\ref{e1}) show that $\hat{r}_{\text{edge}}(B_{m,n}^{\prec}) = \Theta(m 2^m n^2)$. This together with Proposition \ref{pro-1} show that even for sparse graphs like book graphs, its size edge-ordered Ramsey number can not be linear or near-linear. Therefore, we consider Problem \ref{pro2} for some bipartite graphs. 
In Subsection 3.1, we establish that $mK_2^{\prec}$ and $P_4^{\prec}$ exhibit a linear relationship in $m$, namely $\hat{r}_{\text{edge}}(m K_2^{\prec}, P_4^{\prec}) = 3m$ for all integers $m \geq 2$ and any edge-ordered path with $4$ vertices (Theorem \ref{th-3-1}). 
In Subsection 3.2, we prove that $2K_2^{\prec}$ and $K_{s,t}^{\prec}$ satisfy a near-linear relationship in large $t$ (Theorem \ref{th-3-3}). For any integers $m \geq 3$ and $t \geq s \geq 1$, we further derive the upper and lower bounds for $\hat{r}_{\text{edge}}(m K_2^{\prec}, K_{s,t}^{\prec})$. Furthermore, both the upper and lower bounds are tight (Theorem \ref{th-3-5}). 
In Subsection 3.3, we show that the edge-ordered complete bipartite graph $K_{s,t}^{\prec}$ satisfies a linear relationship in $t$ for large $t$, that is, $\hat{r}_{\text{edge}}(K_{1,n}^{\prec}, K_{s,t}^{\prec}) \leq s^2 (n-1) + st$ for any positive integers $n, s, t$, and this upper bound is sharp (Theorem \ref{th-4-1}). For any integers $t \geq s \geq 2$, we have $\hat{r}_{\text{edge}}(K_{s,t}^{\prec}) \leq e s^2 2^{s+2} t$, and this upper bound is sharp for large $t$ (Theorem \ref{th-4-2}).
\end{itemize}

\section{Non-linear results}
\indent
\par
We first introduce some definitions and theorems to be used in the following. 

\begin{definition}
Given a graph $G$ with nonempty, disjoint vertex subsets $X$ and $Y$, the density of the pair $(X,Y)$ is defined as
$$
d(X,Y)=\frac{e(X,Y)}{|X||Y|},
$$
where $e(X,Y)$ denotes the number of edges with one endpoint in $X$ and the other in $Y$.
\end{definition}

We denote this by $d(v,Y)$ when $X=\{v\}$.

\begin{definition}
 Given $\epsilon >0$ and a graph $G$, for two disjoint nonempty vertex subsets $X$ and $Y$, the pair $(X, Y)$ is  called $\epsilon$-regular if for every $A \subseteq X$ and $B \subseteq Y$ such that $|A| \geq \epsilon|X|$ and $|B| \geq \epsilon|Y|$,
$$
|d(A, B)-d(X, Y)| \leq \epsilon.
$$
\end{definition}

In particular, a vertex subset $X$ is said to be \emph{$\epsilon$-regular} if the pair $(X,X)$ is $\epsilon$-regular.

\begin{definition}
An equitable partition of a graph $G$ is a partition of its vertex set $V(G)=V_1\cup\cdots\cup V_k$, where the subsets $V_i$ are nonempty and disjoint, and $\big||V_i|-|V_j|\big|\leq 1$ for all $1\leq i,j\leq k$.
\end{definition}

We appeal to the following strengthened version of Szemer\'{e}di's regularity lemma, whose proof appears in \cite{CFW22}.

\begin{theorem}\upshape{\cite{CFW22}}\label{CFW22}
For any $\epsilon >0$ and integer $m_0 \geq 0$, there exists an integer $M=M(m_0, \epsilon)>m_0$, and for every graph $G$, there exists an equitable partition $V(G)= V_1 \cup \ldots \cup V_k$, which satisfies $m_0 \leq k \leq M$ and 
\begin{itemize}
    \item[] $(1)$ each part $V_i$ is $\epsilon$-regular,
\item[] $(2)$ for every $1 \leq i \leq k$, there are at most $\epsilon k$ values $1 \leq j \leq k$ such that the pair $(V_i, V_j)$ is $\epsilon$-irregular.
\end{itemize}
\end{theorem}

We additionally require the following consequence of the counting lemma, whose proof is given in \cite{CFW22}. Let $Q$ be a copy of clique $K_n$ in a graph $G$. A vertex $u\in V(G)$ is said to \emph{extend $Q$} if $u$ is a common neighbor of all $n$ vertices of $Q$.

\begin{theorem}\upshape{\cite{CFW22}}\label{c22}
Fix an integer $n \geq 2$, and let $\epsilon,\delta\in\bigl(0,\frac{1}{2}\bigr)$ satisfy $\epsilon\le\delta^{3 n^{2}}$. Suppose $X$ is a vertex subset of a graph $G$ which is $\epsilon$-regular and has edge density at least $\delta$. Then $X$ contains at least one $K_{n}$. Furthermore, if $Q$ is a uniformly random $K_{n}$ copy in $X$, then for $ \forall \:
u \in V(G)$,
$$
\operatorname{Pr}\bigl(u \text{ extends }Q\bigr)\ge d(u,X)^{n}-\delta.
$$
\end{theorem}

\begin{definition}\upshape{\cite{CFW23}}\label{de-3-4}
Fix an integer $k\ge2$ and $\epsilon,\delta\in(0,1)$. Let $G$ be a $2$-colored graph and $V_1,\dots,V_k$ be disjoint vertex subsets of  $G$. We define $\{V_1,\dots,V_k\}$ as a $(k,\epsilon,\delta)$-good configuration if it satisfies the following conditions:
\begin{itemize}
    \item[] $(1)$ $V_1\cup\cdots\cup V_k$ induces a complete subgraph of $G$, that is, any two distinct vertices from this union are connected by an edge in $G$.
 \item[] $(2)$ Each $V_i$ is $\epsilon$-regular in the red subgraph, with internal red edge density at least $\delta$.
 \item[] $(3)$ All pairs $(V_i,V_j)$ for $1 \leq i <j \leq k$ are $\epsilon$-regular in the blue subgraph, with blue edge density at least $\delta$.
\end{itemize}
\end{definition}

The following lemma indicates that colorings with $(k,\epsilon,\delta)$-good configurations also contain large monochromatic books, as shown in \cite{CFW23}.

\begin{theorem}\upshape{\cite{CFW23}}\label{CFW23}
Fix $k \geq 2$, $0<\delta \leq 2^{-k-1}$, and $0<\epsilon \leq \delta^{3 k^2}$. Suppose the edges of graph $G$ are $2$-colored, with $V_1, \ldots, V_k$ being a $(k, \epsilon, \delta)$-good configuration in $G$. If the vertices in $V_1 \cup \cdots \cup V_k$ have $t$ common neighbors in $G$, then the coloring contains a monochromatic $B_{k, 2^{-k-1} t}$. 
\end{theorem}

\begin{theorem}\upshape{\cite{BM08}}\label{Turan}
 Given a graph $G$ with $n$ vertices, the maximum number of edges in a $K_m$-free subgraph of $G$ is $\left(1-\frac{1}{m-1}\right)\frac{n^2}{2}$.
\end{theorem}

The proof idea of the following theorem is from \cite{CFW23}.

\begin{theorem}\label{th-5-1}
Let $m$ be an integer with $m \geq 2$. If $n$ is sufficiently large, then 
$$
\hat{r}_{\text{edge}}(B_{m,n}^{\prec}) \leq (2m^2 + m 2^{m+2})n^2-mn.
$$
\end{theorem}

\begin{proof}
 Let $G^{\prec}=B_{2mn, 2^{m+1} n}^{\prec}$. Then $|E(G^{\prec})|=\binom{2mn}{2}+2mn \cdot 2^{m+1} n=(2m^2 + m 2^{m+2})n^2-mn$. Let $\delta=\min \{2^{-m-1}, \frac{1}{10 m^3}\}$, $\epsilon=\delta^{3m^2}$, and $m_0=10 m^2$. Fix a $2$-coloring $\chi$ of $G^{\prec}$, by applying Theorem \ref{CFW22}, we analyze the red graph corresponding to the spine-induced subgraph of $G^{\prec}$. Then there exists an equitable partition $V_1 \cup \ldots \cup V_k$, which satisfies $m_0 \leq k \leq M$ and each part $V_i$ is $\epsilon$-regular, and for every $1 \leq i \leq k$, there are at most $\epsilon k$ values $1 \leq j \leq k$ such that the pair $(V_i, V_j)$ is $\epsilon$-irregular. 

\setcounter{fact}{0}
\begin{fact}
 If $(V_i, V_j)$ is $\epsilon$-regular in red subgraph, then  $(V_i, V_j)$ is also $\epsilon$-regular in blue subgraph. 
\end{fact}

We refer to a part $V_i$ as \emph{red} if at least half of its internal edges are red; otherwise, we term it \emph{blue}. Without loss of generality, assume that at least half of the parts are red. Relabel these red parts as $V_1, \ldots, V_{k'}$ where $5 m^2=\lceil \frac{m_0}{2} \rceil \leq \lceil \frac{k}{2} \rceil \leq k' \leq k$. We construct a reduced graph $F$ with vertex set $\{v_1, \ldots, v_{k'}\}$. An edge $v_i v_j$ exists in $F$ if the pair $(V_i, V_j)$ is $\epsilon$-regular and its blue edge density is at least $\delta$.

\setcounter{case}{0}
\begin{case}
 $\delta(F) > (1-\frac{1}{m-1})k'.$  
\end{case}

As $\delta(F) > (1-\frac{1}{m-1})k'$, we have $|E(F)| > (1-\frac{1}{m-1}) \frac{k'^2}{2}$. Applying Theorem \ref{Turan}, we find that $F$ contains a copy of $K_m$. Let $v_{i_1},\ldots,v_{i_m}$ be the vertices of this $K_m$, and let $V_1 = V_{i_1},\ldots,V_m = V_{i_m}$. 
Observe that all vertices in $V_1\cup\cdots\cup V_m$ lie in the spine of $G^{\prec}$. This union induces a complete subgraph in $G^{\prec}$. In addition, each $V_i$ is $\epsilon$-regular in the red subgraph and its red edge density of at least $\frac{1}{2} \geq \delta$. Furthermore, all pairs $(V_i,V_j)$ with $1 \leq i <j \leq m$ are $\epsilon$-regular in the blue subgraph and its blue edge density of at least $\delta$. We conclude that $\{V_1,\ldots,V_m\}$ forms a $(m,\epsilon,\delta)$-good configuration by Definition \ref{de-3-4}. 
The vertices in $V_1 \cup \cdots \cup V_m$ have $2^{m+1} n$ common neighbors in $G^{\prec}$. By Theorem \ref{CFW23}, it follows that a monochromatic $B_{m, n}^{\prec}$ exists under the coloring $\chi$. 
Given that $B_{m,n}^{\prec}$ is a subgraph of $G^{\prec}$, it follows that its edge order is preserved. In other words, we obtain a monochromatic edge-ordered $B_{m,n}^{\prec}$, as required.

\begin{case}
There exists a vertex in $F$ with degree at most $\bigl(1-\frac{1}{m-1}\bigr)k'$.  
\end{case}

Without loss of generality, we denote this vertex by $v_1$, where $d_F(v_1) \leq \bigl(1-\frac{1}{m-1}\bigr)k'$. Then there are at most $\bigl(1-\frac{1}{m-1}\bigr)k'$ parts $V_j$ (with $2 \leq j \leq k'$) satisfying that the pair $(V_1, V_j)$ is $\epsilon$-regular and has blue edge density of at least $\delta$. As a result, we obtain at least $\frac{k'}{m-1}-1$ parts $V_j$ (with $2 \leq j \leq k'$) for which the pair $(V_1, V_j)$ is either $\epsilon$-irregular or its blue edge density is less than $\delta$. 
Thus, let $U$ denote the set of indices $j$ with $2 \leq j \leq k'$ such that the pair $(V_1, V_j)$ is $\epsilon$-regular and has blue edge density less than $\delta$. From Theorem \ref{CFW22}, the number of $\epsilon$-irregular pairs $(V_1, V_j)$ is no more than $\epsilon k$. Note that $k' \geq 5 m^2$ and $\epsilon \leq \delta \leq \frac{1}{10 m^3}$ for $m \geq 2$. Then 
$$
\begin{aligned}
k-1 \geq k'-1 \geq |U| & \geq\frac{k^{\prime}}{m-1}-1- \epsilon k \\[0.2cm]
& \geq\left(\frac{1}{m-1}-\frac{2}{5 m^2}\right) k^{\prime}\\[0.2cm]
&=\left(\frac{1}{m-1}-\frac{1}{m^2}+\frac{3}{5 m^2}\right) k^{\prime} \\[0.2cm]
&\geq\left(\frac{1}{m}+\frac{3}{5 m^2} \right) k^{\prime}\\[0.2cm]
&\geq\left(\frac{1}{m}+3 m \delta\right) k^{\prime},
\end{aligned}
$$
where the third inequality follows from $\frac{1}{m-1}-\frac{1}{m}=\frac{1}{m(m-1)} \geq \frac{1}{m^2}$.

By Theorem \ref{CFW22}, the partition is equitable, which implies that $|V_i| \geq \left\lfloor \frac{2mn}{k} \right\rfloor \geq \frac{2mn}{k}-1$ for $1 \leq i \leq k$. Let $W=\bigcup_{j \in U} V_j$. Then 
$$
\begin{aligned}
|W| & \geq \left(\frac{1}{m}+3 m \delta\right) k^{\prime} \cdot \left(\frac{2mn}{k}-1 \right)\\[0.2cm]
& \geq \left(\frac{1}{m}+3 m \delta\right)mn-(k-1)\\[0.2cm]
& > \left(\frac{1}{m}+3 m \delta\right)mn-M=\left(\frac{1}{m}+3 m \delta -\frac{M}{mn}\right)mn\\[0.2cm]
&\geq \left(\frac{1}{m}+2 m \delta\right) mn,
\end{aligned}
$$
where the second inequality is derived from $k' \geq \left\lceil \frac{k}{2} \right\rceil \geq \frac{k}{2}$ and $3 m \delta \leq \frac{3}{10 m^2} < \frac{1}{m} $, the third inequality follows from $k \leq M$, and the fourth inequality stems from $\frac{M}{mn} \leq m\delta$ given that $n$ is sufficiently large.

Observe that $V_1$ is $\epsilon$-regular and its red density is at least $\frac{1}{2} \geq \delta$ and $\epsilon=\delta^{3m^2}$. Theorem \ref{c22} guarantees that $V_1$ contains at least one red $K_{m}$. For a uniformly random red $K_{m}$ copy $Q$ in $V_1$, the probability of any vertex $u \in W$ extending $Q$ is no less than $d_R\bigl(u, V_1\bigr)^m - \delta$, with $d_R$ representing the edge density in the red graph. 
We define a random variable $Y_u$ for each vertex $u\in W$, which takes the value 1 if and only if $u$ extends $Q$, and 0 otherwise. Thus, $Y_u$ is a Bernoulli random variable, and
$$
\mathbb{E}[Y_u] = \Pr(Y_u = 1) = \Pr\left(u \text{ extends } Q\right).
$$

As the function $f(x)=x^m$ is convex for $x \geq 0$ and $m\geq 1$, application of Jensen's inequality
$f\left(\frac{1}{n} \sum_{i=1}^n  x_i\right) \leq \frac{1}{n} \sum_{i=1}^n f(x_i),$
gives the inequality
\begin{equation}\label{e5}
\sum_{i=1}^n x_i^m \geq n \left(\frac{1}{n} \sum_{i=1}^n  x_i\right)^m.
\end{equation}
Since $U$ is defined as the set of indices $j$ satisfying $2 \leq j \leq k'$, $(V_1, V_j)$ is $\epsilon$-regular, and its blue edge density of $(V_1, V_j)$ is less than $\delta$, and $W=\bigcup_{j \in U} V_j$, it follows that
\begin{equation}\label{e6}
d_R\left(V_1, V_j\right) \geq 1-\delta    
\end{equation}
for all $j \in U$.

By (\ref{e5}) and (\ref{e6}), and the linearity of expectation, we have
$$
\begin{aligned}
\mathbb{E}\left[\sum_{u \in W} Y_u\right] 
&= \sum_{u \in W} \mathbb{E}[Y_u] = \sum_{u \in W} \Pr(u \text{ extends } Q)\\[0.2cm]
& \geq \sum_{u \in W}\left(d_R\left(u, V_1\right)^m-\delta\right) = \sum_{u \in W} d_R\left(u, V_1\right)^m-\delta |W|\\[0.2cm]
&\geq \left(\frac{1}{|W|} \sum_{u \in W} d_R\left(u, V_1\right)\right)^m |W|-\delta|W|\\[0.2cm]
& \geq \left((1-\delta)^m-\delta\right)|W|\\[0.2cm]
& \geq (1-2 m \delta) \left(\frac{1}{m}+2 m \delta\right) mn\\[0.2cm]
& \geq \left(\frac{1}{m}+2 m \delta-2 m \delta\right) mn = n,
\end{aligned}
$$
where the fourth inequality follows from Bernoulli's inequality $(1+x)^y \geq 1+xy$ for $x \geq -1$ and $y\geq 0$, and the last inequality follows from $2m \delta \leq \frac{1}{5m^2} \leq \frac{m-1}{m}$ for $m\geq 2$.

Consequently, there exists a red $K_m$ in $V_1$ that has at least $n$ red extensions, i.e., a red $B_{m,n}^{\prec}$. Given that $B_{m,n}^{\prec}$ is a subgraph of $G^{\prec}$, it follows that its edge order is preserved. In other words, we obtain a monochromatic edge-ordered $B_{m,n}^{\prec}$, as required.
\end{proof}

For any $m \geq 25$ and every $n \geq 300m^2$, $\hat{r}(B_{m,n}) \geq m 2^m n^2 / 1200$, a result established in \cite{CFW23}. By Theorem \ref{th-5-1} and (\ref{e1}), we have the following corollary:
\begin{corollary}\label{cor1}
$\hat{r}_{\text{edge}}(B_{m,n}^{\prec}) = \Theta(m 2^m n^2)$ for sufficiently large $n$ and a fixed integer $m$.    
\end{corollary}

\section{Linear and near-linear results}

From Corollary \ref{cor1}, both the lower and upper bounds are nonlinear. This implies that even sparse book graphs fail to admit linear bounds, which makes the search for graphs with linear bounds challenging. Hence, our focus is directed toward graph classes such as bipartite graph.

\subsection{Matchings versus small paths}
\indent
\par
Turning to Problem \ref{pro2}, we first derive foundational observations about size edge-ordered Ramsey numbers for matchings and general edge-ordered graphs, then focus on specific families to verify linearity. We begin with two key inequalities that guide our analysis.

Let $G^{\prec}$ be an edge-ordered graph with $|E(G^{\prec})| = q$. For $n \in \mathbb{N}$, $nG^{\prec}$ denotes $n$ vertex-disjoint copies $G_1^{\prec}, \dots, G_n^{\prec}$ of $G^{\prec}$, where each $G_i^{\prec}$ has edge set $\{e_{i1}, \dots, e_{iq}\}$ satisfying $e_{i1} \prec e_{i2} \prec \dots \prec e_{iq}$ for $1 \leq i \leq n.$
The edge order of $nG^{\prec}$ is defined as follows,
$$
e_{ij} \prec e_{kl} \iff i < k \quad \text{or} \quad (i = k \text{ and } j < l).
$$

For edge-ordered graphs $F^{\prec}, G^{\prec}, H^{\prec}$, we define $F^{\prec} \to (G^{\prec}, H^{\prec})$ to mean that any $2$-coloring of the edges of $F^{\prec}$ contains either a red edge-ordered copy of $G^{\prec}$ or a blue edge-ordered copy of $H^{\prec}$.

\begin{proposition}\label{prop-basic-ineq}
For edge-ordered graphs $F^{\prec}, G^{\prec}, H^{\prec}$, we have
\begin{equation}\label{e2}
\hat{r}_{\text{edge}}(m G^{\prec}, n H^{\prec}) \leq (m+n-1) \hat{r}_{\text{edge}} (G^{\prec}, H^{\prec}).    
\end{equation}
Specifically, when $G = K_2$ and $n=1$, we have

\begin{equation}\label{e3}
\max \{m, |E(H^{\prec})|\} \leq \hat{r}_{\text{edge}}(m K_2^{\prec}, H^{\prec}) \leq m \hat{r}_{\text{edge}}  (K_2^{\prec}, H^{\prec})=m|E(H^{\prec})|.    
\end{equation}
\end{proposition}

\begin{proof}
Let $F^{\prec}$ be a edge-ordered graph with $\hat{r}_{\text{edge}}(G^{\prec}, H^{\prec})$ edges such that $F^{\prec} \to (G^{\prec}, H^{\prec})$.
We first prove that $(m+n-1)F^{\prec} \to (mG^{\prec}, nH^{\prec})$.
Recall that $(m+n-1)F^{\prec}$ denotes $m+n-1$ vertex-disjoint copies of $F^{\prec}$, denoted $F_1^{\prec}, F_2^{\prec}, \dots, F_{m+n-1}^{\prec}$, with the edge order defined as follows: for any edges $e \in E(F_i^{\prec})$ and $f \in E(F_j^{\prec})$, $e \prec f$ if and only if $i < j$, or $i = j$ and $e \prec f$ in the original edge order of $F^{\prec}$.

Consider any 2-coloring (red/blue) of the edges of $(m+n-1)F^{\prec}$. For each copy $F_k^{\prec}$ ($1 \leq k \leq m+n-1$), the hypothesis $F^{\prec} \to (G^{\prec}, H^{\prec})$ implies that $F_k^{\prec}$ contains either a red edge-ordered copy of $G^{\prec}$ or a blue edge-ordered copy of $H^{\prec}$. By the pigeonhole principle, among these $m+n-1$ copies, there must exist either at least $m$ copies each containing a red edge-ordered copy of $G^{\prec}$ or at least $n$ copies each containing a blue edge-ordered copy of $H^{\prec}$. 

We verify the edge-order consistency for the red case (the blue case is symmetric). Let $F_{i_1}^{\prec}, \dots, F_{i_m}^{\prec}$ ($i_1 < \dots < i_m$) be the copies containing red $G^{\prec}$ copies, denoted $G_1^{\prec} \subseteq F_{i_1}^{\prec}, \dots, G_m^{\prec} \subseteq F_{i_m}^{\prec}$. For any edges $e \in G_a^{\prec}$ and $f \in G_b^{\prec}$ with $a < b$, since $i_a < i_b$, we have $e \prec f$ by the order of $(m+n-1)F^{\prec}$. For edges within $G_a^{\prec}$, their order inherits the original edge-order of $G^{\prec}$. Thus, $G_1^{\prec} \cup \dots \cup G_m^{\prec}$ is a red edge-ordered copy of $mG^{\prec}$. Hence, $(m+n-1)F^{\prec} \to (mG^{\prec}, nH^{\prec})$.

Next, we derive Inequality \eqref{e2}. From the transitivity result above, $(m+n-1)F^{\prec} \to (mG^{\prec}, nH^{\prec})$, then
$$
\hat{r}_{\text{edge}}(mG^{\prec}, nH^{\prec}) \leq (m+n-1) \hat{r}_{\text{edge}}(G^{\prec}, H^{\prec}).
$$

We now focus on the special case where $G = K_2^{\prec}$ and $n=1$, leading to Inequality \eqref{e3}. For the lower bound, to contain a monochromatic $mK_2^{\prec}$ or a monochromatic $H^{\prec}$, the graph must have at least $m$ or $|E(H^{\prec})|$ edges, respectively. Thus, the lower bound is the maximum of these two values.

Note that $\hat{r}_{\text{edge}}(K_2^{\prec}, H^{\prec}) = |E(H^{\prec})|$. Substituting $G = K_2^{\prec}$ and $n=1$ into Inequality \eqref{e2}, we get
$$
\hat{r}_{\text{edge}}(mK_2^{\prec}, H^{\prec}) \leq (m+1-1) \hat{r}_{\text{edge}}(K_2^{\prec}, H^{\prec}) = m|E(H^{\prec})|.
$$
Combining the lower and upper bounds, we have
$$
\max\{m, |E(H^{\prec})|\} \leq \hat{r}_{\text{edge}}(mK_2^{\prec}, H^{\prec}) \leq m|E(H^{\prec})|.
$$
\end{proof}

A key link to size-Ramsey numbers is as follows.

\begin{proposition}\label{prop-1}
 If $\hat{r}\bigl(m K_2, H\bigr) = m\bigl|E(H)\bigr|$, then $\hat{r}_{\text{edge}}\bigl(m K_2^{\prec}, H^{\prec}\bigr) = m\bigl|E\bigl(H^{\prec}\bigr)\bigr|$  for any edge-ordered graph $H^{\prec}$.
\end{proposition}
\begin{proof}
By (\ref{e1}), we have
$$
\hat{r}_{\text {edge }}\left(m K_2^{\prec}, H^{\prec}\right) \geq \hat{r}\left(m K_2, H\right)=m|E(H)|=m\left|E\left(H^{\prec}\right)\right| .
$$
Combined with (\ref{e3}), it follows that $\hat{r}_{\text {edge }}\left(m K_2^{\prec}, H^{\prec}\right)=m\left|E\left(H^{\prec}\right)\right|$.
\end{proof}
Let $P_4^{\prec_1}$, $P_4^{\prec_2}$, $P_4^{\prec_3}$ denote edge-ordered paths with vertex set $\{v_i: 1\leq i\leq 4\}$ and edge set $\{e_i=v_iv_{i+1}: 1\leq i\leq 3\}$, where the edge orderings are given by
$$
P_4^{\prec_1}: e_1 \prec e_2 \prec e_3, \quad
P_4^{\prec_2}: e_1 \prec e_3 \prec e_2, \quad
P_4^{\prec_3}: e_2 \prec e_1 \prec e_3.
$$

To establish linearity for $m K_2^{\prec}$ and $P_4^{\prec}$, where $P_4^{\prec}\in \{P_4^{\prec_1}, P_4^{\prec_2},P_4^{\prec_3}\}$, we first analyze the base case $m=2$.

\begin{lemma}\label{le-3-1}
For any edge-ordered $P_4^{\prec}\in \{P_4^{\prec_1}, P_4^{\prec_2},P_4^{\prec_3}\}$, we have $\hat{r}_{\text{edge}}(2K_2^{\prec},P_4^{\prec})=6.$  
\end{lemma}

\begin{proof}
For the upper bound, by Inequality \eqref{e3}, we have 
$\hat{r}_{\text{edge}}(2K_2^{\prec}, P_4^{\prec}) \leq 2|E(P_4^{\prec})| \leq 6.$

For the lower bound, we show that every edge-ordered graph $G^{\prec}$ with $|E(G^{\prec})| = 5$ exists a red/blue coloring avoiding both a red edge-ordered $2K_2^{\prec}$ and a blue edge-ordered $P_4^{\prec}$. We first consider a special case.
\setcounter{case}{0}
\begin{case}
$G^{\prec}$ does not contain $P_4^{\prec}$ as an edge-ordered subgraph.     
\end{case}

Color all edges of $G^{\prec}$ blue. In this coloring, there are no red edges, so there cannot exist a red edge-ordered $2K_2^{\prec}$.
Since $G^{\prec}$ has no edge-ordered $P_4^{\prec}$ subgraph, there is no blue edge-ordered $P_4^{\prec}$.

\begin{case}
$G^{\prec}$ contains $P_4^{\prec}$ as an edge-ordered subgraph.    
\end{case}

We classify $G^{\prec}$ by the longest path in its underlying graph. Assume the longest path is $P_6$ (Figure (a)): Color edges $v_1v_2$ and $v_2v_3$ red, and the remaining edges blue. The red edges are adjacent (sharing $v_2$), so there is no red edge-ordered $2K_2^{\prec}$. The blue edges form $P_2^{\prec} \cup P_3^{\prec}$, which cannot be isomorphic to $P_4^{\prec}$.

Assume the longest path is $P_5$ or $P_4$. These graphs are shown in Figures (b)-(h) (longest path $P_5$) and (i)-(o) (longest path $P_4$). For Figures (c)-(o), color edges $v_1v_2$ and $v_2v_3$ red, and the remaining edges blue. The red edges are adjacent (sharing $v_2$), so there is no red $2K_2^{\prec}$. The underlying graph of the blue edges contains no $P_4$ subgraph, so there is no blue $P_4^{\prec}$.

\tikzset{
    fill-node/.style={fill, circle, inner sep=1.5pt},
    draw-node/.style={draw, circle, inner sep=1.5pt}
}

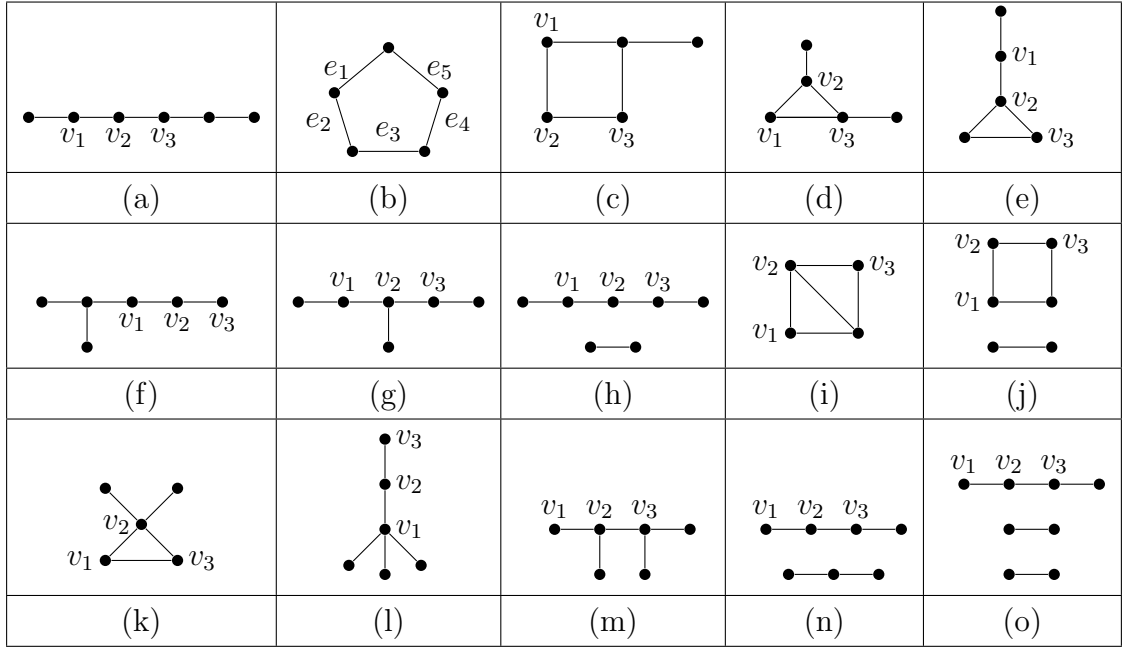
\begin{figure}[H]
    \centering
    \begin{tabular}{|c|c|c|c|c|}
    \hline
   \begin{tikzpicture}[scale=0.6] 
        \node[fill-node] (A) at (0,0) {};
        \node[fill-node] (B) at (1,0) {};
        \node[fill-node] (C) at (2,0) {};
        \node[fill-node] (D) at (3,0) {};
        \node[fill-node] (E) at (4,0) {};
        \node[fill-node] (F) at (5,0) {};
        \draw (A) -- (B) -- (C) -- (D)--(E) -- (F);
    \node[below] at (B) {$v_1$};
    \node[below] at (C) {$v_2$};
    \node[below] at (D) {$v_3$};
    \end{tikzpicture}
    &
    \begin{tikzpicture}[scale=0.6] 
        \node[fill-node] (A) at (0,1.5) {};
        \node[fill-node] (B) at (1.2,0.5) {};
        \node[fill-node] (C) at (0.8,-0.8) {};
        \node[fill-node] (D) at (-0.8,-0.8) {};
        \node[fill-node] (E) at (-1.2,0.5) {};
        \draw (A) -- (B) node[midway, right] {$e_5$};
        \draw (B) -- (C) node[midway, right] {$e_4$};
        \draw (C) -- (D) node[midway, above] {$e_3$};
        \draw (D) -- (E) node[midway, left] {$e_2$};
        \draw (E) -- (A) node[midway, left] {$e_1$};
    \end{tikzpicture}
    &
    \begin{tikzpicture}
    \node[fill-node] (v1) at (0,1) {};
    \node[fill-node] (v2) at (1,1) {};
    \node[fill-node] (v3) at (1,0) {};
    \node[fill-node] (v4) at (0,0) {};
    \node[fill-node] (v5) at (2,1) {};
    \draw (v1) -- (v2) -- (v3) -- (v4) -- (v1);
     \draw (v2) -- (v5);
    \node[above] at (v1) {$v_1$};
    \node[below] at (v4) {$v_2$};
    \node[below] at (v3) {$v_3$};
\end{tikzpicture}
    &
    \begin{tikzpicture}[scale=0.6] 
        \node[fill-node] (A) at (0,0) {};
        \node[fill-node] (B) at (-0.8,-0.8) {};
        \node[fill-node] (C) at (0.8,-0.8) {};
        \node[fill-node] (D) at (0,0.8) {};
        \node[fill-node] (E) at (2,-0.8) {};
        \draw (A) -- (B) -- (C) -- (A);
        \draw (A) -- (D);
         \draw (B) -- (E);
    \node[right] at (A) {$v_2$};
    \node[below] at (B) {$v_1$};
    \node[below] at (C) {$v_3$};
    \end{tikzpicture}
    &
    \begin{tikzpicture}[scale=0.6] 
        \node[fill-node] (A) at (0,0) {};
        \node[fill-node] (B) at (-0.8,-0.8) {};
        \node[fill-node] (C) at (0.8,-0.8) {};
        \node[fill-node] (D) at (0,1) {};
        \node[fill-node] (E) at (0,2) {};
        \draw (A) -- (B) -- (C) -- (A);
        \draw (A) -- (D) -- (E);
    \node[right] at (D) {$v_1$};
    \node[right] at (A) {$v_2$};
    \node[right] at (C) {$v_3$};
    \end{tikzpicture}
    \\
    \hline
    (a) & (b) & (c) & (d) & (e) \\
    \hline
    \begin{tikzpicture}[scale=0.6] 
        \node[fill-node] (A) at (0,0) {};
        \node[fill-node] (B) at (1,0) {};
        \node[fill-node] (C) at (2,0) {};
        \node[fill-node] (D) at (3,0) {};
        \node[fill-node] (E) at (1,-1) {};
         \node[fill-node] (F) at (4,0) {};
        \draw (A) -- (B) -- (C) -- (D)-- (F);
        \draw (B) -- (E);
    \node[below] at (C) {$v_1$};
    \node[below] at (D) {$v_2$};
    \node[below] at (F) {$v_3$};
    \end{tikzpicture}
    &
    \begin{tikzpicture}[scale=0.6] 
        \node[fill-node] (A) at (0,0) {};
        \node[fill-node] (B) at (1,0) {};
        \node[fill-node] (C) at (2,0) {};
        \node[fill-node] (D) at (3,0) {};
        \node[fill-node] (E) at (2,-1) {};
        \node[fill-node] (F) at (4,0) {};
        \draw (A) -- (B) -- (C) -- (D)-- (F);
        \draw (C) -- (E);
    \node[above] at (B) {$v_1$};
    \node[above] at (C) {$v_2$};
    \node[above] at (D) {$v_3$};
    \end{tikzpicture}
    &
    \begin{tikzpicture}[scale=0.6] 
        \node[fill-node] (A) at (0,0) {};
        \node[fill-node] (B) at (1,0) {};
        \node[fill-node] (C) at (2,0) {};
        \node[fill-node] (D) at (3,0) {};
        \node[fill-node] (E) at (1.5,-1) {};
        \node[fill-node] (F) at (2.5,-1) {};
        \node[fill-node] (G) at (4,0) {};
        \draw (A) -- (B) -- (C) -- (D)-- (G);
        \draw (E) -- (F);
    \node[above] at (B) {$v_1$};
    \node[above] at (C) {$v_2$};
    \node[above] at (D) {$v_3$};
    \end{tikzpicture}
    &
    \begin{tikzpicture}[scale=0.6] 
        \node[fill-node] (A) at (0,1.5) {};
        \node[fill-node] (B) at (1.5,1.5) {};
        \node[fill-node] (C) at (1.5,0) {};
        \node[fill-node] (D) at (0,0) {};
        \draw (A) -- (B) -- (C) -- (D) -- (A);
        \draw (A) -- (C);
    \node[left] at (D) {$v_1$};
    \node[left] at (A) {$v_2$};
    \node[right] at (B) {$v_3$};
    \end{tikzpicture}
    &
    \begin{tikzpicture}[scale=0.6] 
        \node[fill-node] (A) at (0,1.3) {};
        \node[fill-node] (B) at (1.3,1.3) {};
        \node[fill-node] (C) at (1.3,0) {};
        \node[fill-node] (D) at (0,0) {};
        \node[fill-node] (E) at (0,-1) {};
        \node[fill-node] (F) at (1.3,-1) {};
        \draw (A) -- (B) -- (C) -- (D) -- (A);
        \draw (E) -- (F);
    \node[left] at (D) {$v_1$};
    \node[left] at (A) {$v_2$};
    \node[right] at (B) {$v_3$};
    \end{tikzpicture}
    \\
    \hline
    (f) & (g) & (h) & (i) & (j) \\
    \hline
    \begin{tikzpicture}[scale=0.6] 
        \node[fill-node] (A) at (0,0) {};
        \node[fill-node] (B) at (-0.8,-0.8) {};
        \node[fill-node] (C) at (0.8,-0.8) {};
        \node[fill-node] (D) at (-0.8,0.8) {};
        \node[fill-node] (E) at (0.8,0.8) {};
        \draw (A) -- (B) -- (C) -- (A);
        \draw (D) -- (A);
        \draw (E) -- (A);
    \node[left] at (B) {$v_1$};
    \node[left] at (A) {$v_2$};
    \node[right] at (C) {$v_3$};
    \end{tikzpicture}
    &
    \begin{tikzpicture}[scale=0.6] 
        \node[fill-node] (A) at (0,0) {};
        \node[fill-node] (B) at (0,-1) {};
        \node[fill-node] (C) at (0,1) {};
        \node[fill-node] (D) at (0,2) {};
        \node[fill-node] (E) at (0.8,-0.8) {};
        \node[fill-node] (F) at (-0.8,-0.8) {};
        \draw (A) -- (B);
        \draw (A) -- (C) -- (D);
        \draw (A) -- (E);
        \draw (A) -- (F);
    \node[right] at (A) {$v_1$};
    \node[right] at (C) {$v_2$};
    \node[right] at (D) {$v_3$};
    \end{tikzpicture}
    &
    \begin{tikzpicture}[scale=0.6] 
        \node[fill-node] (A) at (0,0) {};
        \node[fill-node] (B) at (1,0) {};
        \node[fill-node] (C) at (2,0) {};
        \node[fill-node] (D) at (1,-1) {};
        \node[fill-node] (E) at (2,-1) {};
         \node[fill-node] (F) at (3,0) {};
        \draw (A) -- (B) -- (C)-- (F);
        \draw (B) -- (D);
        \draw (C) -- (E);
    \node[above] at (A) {$v_1$};
    \node[above] at (B) {$v_2$};
    \node[above] at (C) {$v_3$};
    \end{tikzpicture}
    &
    \begin{tikzpicture}[scale=0.6] 
        \node[fill-node] (A) at (0,0) {};
        \node[fill-node] (B) at (1,0) {};
        \node[fill-node] (C) at (2,0) {};
        \node[fill-node] (D) at (3,0) {};
        \node[fill-node] (E) at (0.5,-1) {};
        \node[fill-node] (F) at (1.5,-1) {};
        \node[fill-node] (G) at (2.5,-1) {};
        \draw (A) -- (B) -- (C) -- (D);
        \draw (E) -- (F)-- (G);
    \node[above] at (A) {$v_1$};
    \node[above] at (B) {$v_2$};
    \node[above] at (C) {$v_3$};
    \end{tikzpicture}
    &
    \begin{tikzpicture}[scale=0.6] 
        \node[fill-node] (A) at (0,0) {};
        \node[fill-node] (B) at (1,0) {};
        \node[fill-node] (C) at (2,0) {};
        \node[fill-node] (D) at (3,0) {};
        \node[fill-node] (E) at (1,-1) {};
        \node[fill-node] (F) at (2,-1) {};
        \node[fill-node] (G) at (1,-2) {};
        \node[fill-node] (H) at (2,-2) {};
        \draw (A) -- (B) -- (C) -- (D);
        \draw (E) -- (F);
        \draw (G) -- (H);
    \node[above] at (A) {$v_1$};
    \node[above] at (B) {$v_2$};
    \node[above] at (C) {$v_3$};
    \end{tikzpicture}
    \\
    \hline
    (k) & (l) & (m) & (n) & (o) \\
    \hline
    \end{tabular}
    \caption{Some graphs with $5$ edges}
    \label{fig:graphs}
\end{figure}

Note that $G^{\prec} = C_5^{\prec}$ (Figure (b)) contains $P_4^{\prec}$. We analyze this case based on the edge-order of $P_4^{\prec}$. Assume that $P_4^{\prec}=P_4^{\prec_1}$. Without loss of generality, let $e_1 \prec e_2 \prec e_3$. If $e_4 \prec e_3$, color edges $e_1$ and $e_5$ red, and the remaining edges blue. The blue edges satisfy the edge-order $e_2 \prec e_3$ and $e_4 \prec e_3$. If $e_3 \prec e_4$ and $e_5 \prec e_4$, color edges $e_1$ and $e_2$ red, and the remaining edges blue. The blue edges satisfy the edge-order $e_3 \prec e_4$ and $e_5 \prec e_4$. If $e_3 \prec e_4 \prec e_5$, color edges $e_3$ and $e_4$ red, and the remaining edges blue. The blue edges satisfy the edge-order $e_1 \prec e_2 \prec e_5$.

Assume that $P_4^{\prec}=P_4^{\prec_2}$. Without loss of generality, let $e_1 \prec e_3 \prec e_2$ (where the middle edge $e_2$ has the maximum order). Color edges $e_1$ and $e_5$ red, and the remaining edges blue. Note that $e_3 \prec e_2$, so the middle edge $e_3$ does not have the maximum order in the blue subgraph.

Similarly, assume that $P_4^{\prec}=P_4^{\prec_3}$ . Without loss of generality, let $e_2 \prec e_1 \prec e_3$ (where the middle edge $e_1$ has the minimum order). Color edges $e_1$ and $e_5$ red, and the remaining edges blue. Note that $e_2 \prec e_3$, so the middle edge $e_1$ does not have the minimum order in the blue subgraph.

In all the above colorings, there is no red edge-ordered $2K_2^{\prec}$ and no blue edge-ordered $P_4^{\prec}$. Thus, $\hat{r}_{\text{edge}}(2K_2^{\prec}, P_4^{\prec}) \geq 6$.
\end{proof}

Using Lemma \ref{le-3-1} as the base case, we generalize to arbitrary $m$ via induction, establishing linearity for $m K_2^{\prec}$ and edge-ordered $P_4^{\prec}$.

\begin{theorem}\label{th-3-1}
For any integer $m \geq 2$ and any edge-ordered $P_4^{\prec}\in \{P_4^{\prec_1}, P_4^{\prec_2},P_4^{\prec_3}\}$, we have 
$$\hat{r}_{\text{edge}}(mK_2^{\prec}, P_4^{\prec}) = 3m.$$       
\end{theorem}

\begin{proof}
For the upper bound, we have $\hat{r}_{\text{edge}}(mK_2^{\prec}, P_4^{\prec}) \leq 3m$ by (\ref{e3}). For the lower bound, let 
$$
\mathcal{R}_{\text{edge}}(mK_2^{\prec}, P_4^{\prec}) = \{ G^{\prec} : G^{\prec} \to (mK_2^{\prec}, P_4^{\prec}) \text{ but } G'^{\prec} \nrightarrow (mK_2^{\prec}, P_4^{\prec}) \text{ for any } G'^{\prec} \varsubsetneq G^{\prec} \}.
$$
Take any minimal Ramsey subgraph $G^{\prec} \in \mathcal{R}_{\text{edge}}(mK_2^{\prec}, P_4^{\prec})$. We prove $|E(G^{\prec})| \geq 3m$ by induction on $m$.
For $m=2$, this follows directly from Lemma \ref{le-3-1}. Assume that it holds for $m-1$, then we consider the case for $m$ below. 

\begin{claim}\label{c1}
For any vertex $v \in V(G^{\prec})$, we have $G^{\prec} - v \to ((m-1)K_2^{\prec}, P_4^{\prec})$ and $|E(G^{\prec} - v)| \geq 3(m-1)$.  
\end{claim}

\begin{proof}
Assume for contradiction that there exists a vertex $v \in V(G^{\prec})$ such that $G^{\prec} - v \nrightarrow ((m-1)K_2^{\prec}, P_4^{\prec})$, that is, there exists a red/blue coloring $\chi$ of $G^{\prec} - v$ such that $G^{\prec} - v$ contains neither a red edge-ordered $(m-1)K_2^{\prec}$ nor a blue edge-ordered $P_4^{\prec}$.

Extend $\chi$ to a coloring of $G^{\prec}$ by coloring all edges incident to $v$ red. We verify that this extension avoids both forbidden subgraphs. Since $\chi$ contains no blue edge-ordered $P_4^{\prec}$ in $G^{\prec} - v$, and all edges incident to $v$ are red, it follows that there is no blue edge-ordered $P_4^{\prec}$ in $G^{\prec}$.
Furthermore, $\chi$ contains at most $m-2$ pairwise disjoint red edges in $G^{\prec} - v$. The red edges incident to $v$ all share the vertex $v$, so they cannot form additional pairwise disjoint red edges with those in $G^{\prec} - v$. Thus, $G^{\prec}$ contains at most $m-1$ pairwise disjoint red edges. This contradicts the assumption that $G^{\prec} \to (mK_2^{\prec}, P_4^{\prec})$. 

Let $H^{\prec}$ be a minimal Ramsey subgraph of $G^{\prec} - v$ for $((m-1)K_2^{\prec}, P_4^{\prec})$. By the induction hypothesis, $|E(H^{\prec})| \geq 3(m-1)$. Since $H^{\prec} \subseteq G^{\prec} - v$, we have $|E(G^{\prec} - v)| \geq |E(H^{\prec})| \geq 3(m-1)$. Hence, Claim \ref{c1} holds.
\end{proof}

\setcounter{case}{0} 
\begin{case}
There exists a vertex $v \in V(G^{\prec})$ with $d_{G^{\prec}}(v) \geq 3$.    
\end{case}

By Claim \ref{c1}, $|E(G^{\prec} - v)| \geq 3(m-1)$. Thus,
$$
|E(G^{\prec})| = d_{G^{\prec}}(v) + |E(G^{\prec} - v)| \geq 3 + 3(m-1) = 3m.
$$

\begin{case}
Every vertex $v \in V(G^{\prec})$ has $d_{G^{\prec}}(v) \leq 2$.    
\end{case}

If every vertex $v \in V(G^{\prec})$ has $d_{G^{\prec}}(v) \leq 2$, then $G^{\prec}$ is a disjoint union of edge-ordered paths and cycles. Since $G^{\prec} \to (mK_2^{\prec}, P_4^{\prec})$, it follows that each connected component of $G^{\prec}$ must contain $P_4^{\prec}$ as an edge-ordered subgraph. Otherwise, suppose some components do not contain $P_4^{\prec}$ as an edge-ordered subgraph. Color all edges of these components blue, and color the edges of the remaining components arbitrarily. In this coloring, any red edge-ordered $mK_2^{\prec}$ or blue edge-ordered $P_4^{\prec}$ must lie entirely within the remaining components (as the blue-colored components contain no $P_4^{\prec}$). However, the subgraph induced by these remaining components (denoted $G'^{\prec}$) satisfies $G'^{\prec} \varsubsetneq G^{\prec}$, which contradicts the minimality assumption that $G'^{\prec} \nrightarrow (mK_2^{\prec}, P_4^{\prec})$ for all $G'^{\prec} \varsubsetneq G^{\prec}$. We analyze the components by their type below.

\begin{claim}\label{c2}
Let any edge-ordered $P_3^{\prec}$, say $V(P_3^{\prec}) = \{u_1, u_2, u_3\}$. If $G^{\prec} - v = P_3^{\prec} \cup (G^{\prec} - \{v, u_2\})$ and $G^{\prec} - v \to ((m-1)K_2^{\prec}, P_4^{\prec})$ for some vertex $v \in V(G^{\prec})$, then $G^{\prec} - \{v, u_2\} \to ((m-1)K_2^{\prec}, P_4^{\prec})$.   
\end{claim}

\begin{proof}
Assume, for contradiction, that there exists a 2-coloring $\chi'$ of $G^{\prec} - \{v, u_2\}$ such that $G^{\prec} - \{v, u_2\} \nrightarrow ((m-1)K_2^{\prec}, P_4^{\prec})$. Color the edges $u_1u_2$ and $u_2u_3$ blue, and keep the coloring of $\chi'$ in $G^{\prec} - \{v, u_2\}$ unchanged. Since $G^{\prec} - \{v, u_2\} \nrightarrow ((m-1)K_2^{\prec}, P_4^{\prec})$ under $\chi'$, it follows that $P_3^{\prec} \cup (G^{\prec} - \{v, u_2\}) \nrightarrow ((m-1)K_2^{\prec}, P_4^{\prec})$. Since $G^{\prec} - v = P_3^{\prec} \cup (G^{\prec} - \{v, u_2\})$, it follows that $G^{\prec} - v \nrightarrow ((m-1)K_2^{\prec}, P_4^{\prec})$, a contradiction.
\end{proof}

\begin{subcase}
$G^{\prec}$ contains a path component $P_n^{\prec}$ for $n \geq 4$.    
\end{subcase}

Let $V(P_n^{\prec}) = \{v_1, v_2, v_3, \ldots, v_n\}$. By Claim \ref{c1}, we have $G^{\prec} - v_3 \to ((m-1)K_2^{\prec}, P_4^{\prec})$. Note that $G^{\prec} - v_3 = (G^{\prec} - \{v_2, v_3\}) \cup \{v_1v_2\}$. Using the same argument as in Claim \ref{c2}, we obtain $G^{\prec} - \{v_2, v_3\} \to ((m-1)K_2^{\prec}, P_4^{\prec})$. By Claim \ref{c1}, it follows that $|E(G^{\prec} - \{v_2, v_3\})| \geq 3(m-1)$. Therefore,
$$
|E(G^{\prec})| = d_{G^{\prec}}(v_2) + d_{G^{\prec}}(v_3) - 1 + |E(G^{\prec} - \{v_2, v_3\})| \geq 3 + 3(m-1) = 3m.
$$

\begin{subcase}
 $G^{\prec}$ is a graph composed of edge-ordered cycles, that is, $G^{\prec} = C_{n_1}^{\prec_1} \cup C_{n_2}^{\prec_2} \cup \ldots \cup C_{n_k}^{\prec_k}$, then $n_i \geq 4$ for each $1 \leq i \leq k$.    
\end{subcase}

Suppose that there exists some $i$ with $1 \leq i \leq k$ such that $n_i = 4$. Let $V(C_4^{\prec}) = \{v_1, v_2, v_3, v_4\}$. By Claim \ref{c1}, we have $G^{\prec} - v_1 \to ((m-1)K_2^{\prec}, P_4^{\prec})$. Note that $G^{\prec} - v_1 = \{v_2v_3, v_3v_4\} \cup (G^{\prec} - \{v_2, v_3\})$. By Claim \ref{c2}, it follows that $G^{\prec} - \{v_2, v_3\} \to ((m-1)K_2^{\prec}, P_4^{\prec})$. By Claim \ref{c1}, we have $|E(G^{\prec} - \{v_2, v_3\})| \geq 3(m-1)$. Thus,
$$
|E(G^{\prec})| = d_{G^{\prec}}(v_1) + d_{G^{\prec}}(v_3) + |E(G^{\prec} - \{v_2, v_3\})| \geq 4 + 3(m-1) > 3m.
$$

Suppose that there exists some $i$ with $1 \leq i \leq k$ such that $n_i \geq 6$. Let $V(C_{n_i}^{\prec}) = \{v_1, v_2, \ldots, v_{n_i}\}$. By Claim \ref{c1}, we have $G^{\prec} - v_{n_i} \to ((m-1)K_2^{\prec}, P_4^{\prec})$. Similarly, by Claim \ref{c1}, we have $G^{\prec} - \{v_{n_i}, v_4\} \to ((m-2)K_2^{\prec}, P_4^{\prec})$. Note that $G^{\prec} - \{v_{n_i}, v_4\} = \{v_1v_2, v_2v_3\} \cup (G^{\prec} - \{v_{n_i}, v_4, v_2\})$. Therefore, by Claim \ref{c2}, it follows that $G^{\prec} - \{v_{n_i}, v_4, v_2\} \to ((m-2)K_2^{\prec}, P_4^{\prec})$. By Claim \ref{c1}, we have $|E(G^{\prec} - \{v_{n_i}, v_4, v_2\})| \geq 3(m-2)$, and hence
$$
|E(G^{\prec})| = d_{G^{\prec}}(v_{n_i}) + d_{G^{\prec}}(v_4) + d_{G^{\prec}}(v_2) + |E(G^{\prec} - \{v_{n_i}, v_4, v_2\})| \geq 6 + 3(m-2) = 3m.
$$

If $n_i = 5$ for each $i$ ($1 \leq i \leq k$), then $G^{\prec} = \underbrace{C_5^{\prec_1} \cup \ldots \cup C_5^{\prec_k}}_k$. Assume that $|E(G^{\prec})| = 3m - 1$. Then $5k = 3m - 1$.

Color each $C_5^{\prec_i}$ according to the coloring in Lemma \ref{le-3-1} (see Figure (b)) such that $C_5^{\prec_i} \nrightarrow (2K_2^{\prec}, P_4^{\prec})$, where $1 \leq i \leq k$. Under this coloring, there are at most $k$ pairwise disjoint red edges and no blue edge-ordered $P_4^{\prec}$. Since $k = \frac{3m - 1}{5} < m$ for $m \geq 2$, it follows that $G^{\prec} \nrightarrow (mK_2^{\prec}, P_4^{\prec})$, a contradiction. Thus, $|E(G^{\prec})| \geq 3m$, as desired.
\end{proof}

\subsection{Small matchings versus complete bipartite graph}
\indent
\par

Using classical size-Ramsey results for bipartite graphs, we establish near-linearity for $2K_2^{\prec}$ and $K_{s,t}^{\prec}$ in terms of $t$.

\begin{theorem}{\upshape \cite{EF81}}\label{EF81}
For any integer $m \geq 1$ and $t\geq s \geq 2$, we have $\hat{r}(2 K_2,K_{s,t})=st+s+t$.
\end{theorem}

Leveraging Theorem \ref{EF81} and Proposition \ref{prop-1}, we extend this result to the edge-ordered setting.

\begin{theorem}\label{th-3-3}
For any integers $t \geq s \geq 2$, we have $\hat{r}_{\text{edge}}(2K_2^{\prec}, K_{s,t}^{\prec}) = st + s + t$.
\end{theorem}

\begin{proof}
By (\ref{e1}) and Theorem \ref{EF81}, we have $\hat{r}_{\text{edge}}(2K_2^{\prec}, K_{s,t}^{\prec}) \geq \hat{r}(2K_2, K_{s,t}) = st + s + t$. For the upper bound, let $F^{\prec} = K_{s+1,t+1}^{\prec} - \{u_1v_1\}$ with two partite sets $A$ and $B$ where $V(A) = \{u_1, \ldots, u_{s+1}\}$ and $V(B) = \{v_1, \ldots, v_{t+1}\}$.

Assume there is no red edge-ordered $2K_2^{\prec}$ and we need to show there exists a blue edge-ordered $K_{s,t}^{\prec}$. Since there is no red edge-ordered $2K_2^{\prec}$, the red subgraph must be a star (all red edges share a common center, otherwise, two edges would be disjoint).

If the center of this red star is in $A$, namely $u_i$ for some $1 \leq i \leq s+1$, then the subgraph induced by $A - \{u_i\}$ and $B - \{v_1\}$ is a blue edge-ordered $K_{s,t}^{\prec}$, and inherits the vertex lexicographical order. Similarly, if the center of this red star is in $B$, namely $v_j$ for some $1 \leq j \leq t+1$, then the subgraph induced by $A - \{u_1\}$ and $B - \{v_j\}$ is a blue edge-ordered $K_{s,t}^{\prec}$.

Thus, $\hat{r}_{\text{edge}}(2K_2^{\prec}, K_{s,t}^{\prec}) \leq st + s + t$. Combining the upper and lower bounds, we conclude $\hat{r}_{\text{edge}}(2K_2^{\prec}, K_{s,t}^{\prec}) = st + s + t$.
\end{proof}

Finally, we generalize to arbitrary $m$ for $K_{s, t}^{\prec}$, establishing tight bounds that confirm linearity in specific cases.

\begin{theorem}{\upshape \cite{FSS80}}\label{FSS80}
For any integers $m \geq 1$ and $t \geq s \geq 1$, we have 
$$r(mK_2, K_{s,t}) = \max \{m + s + t - 1, 2m + s - 1\}.$$
\end{theorem}

\begin{theorem}\label{th-3-5}
For any integers $m \geq 3$ and $t \geq s \geq 1$, we have 
$$
\max \left\{m, st, cs - \binom{m-1}{2}\right\} \leq \hat{r}_{\text{edge}}(mK_2^{\prec}, K_{s,t}^{\prec}) \leq \min \left\{mst, m^2 + m(s + t - 2) + (s-1)(t-1)\right\},
$$
where $c = t + m - 1$ if $t \geq m$ and $c = 2m - 2$ if $t \leq m$. Moreover, the bounds are sharp. 
\end{theorem}

\begin{proof}
We first consider the upper bound. If $m^2 + m(s + t - 2) + (s-1)(t-1) \geq mst$, then $\hat{r}_{\text{edge}}(mK_2^{\prec}, K_{s,t}^{\prec}) \leq mst$ by (\ref{e3}).

Suppose that $mst \geq m^2 + m(s + t - 2) + (s-1)(t-1)$. 
Let $F^{\prec} = K_{m+s-1, m+t-1}^{\prec}$ with partite sets $A$ (of size $m+s-1$) and $B$ (of size $m+t-1$). Then 
$$
|E(F^{\prec})| = (m+s-1)(m+t-1) = m^2 + m(s + t - 2) + (s-1)(t-1).
$$
It suffices to show that any red/blue coloring of $F^{\prec}$ contains either a red edge-ordered $mK_2^{\prec}$ or a blue edge-ordered $K_{s,t}^{\prec}$. Assume there is no red edge-ordered $mK_2^{\prec}$ and no blue edge-ordered $K_{s,t}^{\prec}$; we will derive a contradiction.

Since there is no blue edge-ordered $K_{s,t}^{\prec}$, $F^{\prec}$ contains at least one red edge. By K\"{o}nig's Theorem, in a bipartite graph, the size of a maximum matching equals the size of a minimum vertex cover. Without loss of generality, let $M$ be a maximum red matching of $F^{\prec}$ with $|M| \leq m-1$. Then there exists a minimum vertex cover $C$ of the red subgraph such that $|C| = |M| \leq m-1$. 
Let $A_1 = A \cap C$, $A_2 = A - A_1$, $B_1 = B \cap C$, and $B_2 = B - B_1$. Then $|A_1| + |B_1| = |C| \leq m-1$, so $|A_1| \leq m-1$ and $|B_1| \leq m-1$. 

All edges between $A_2$ and $B_2$ are blue. Otherwise, if there exists a red edge $xy$ with $x \in A_2$ and $y \in B_2$, then $xy$ would not be covered by $C$ (since $x \notin C$ and $y \notin C$), contradicting that $C$ is a vertex cover of the red subgraph. Thus, the subgraph induced by $A_2$ and $B_2$ is a blue edge-ordered complete bipartite graph $K_{|A_2|, |B_2|}^{\prec}$. 

Since $A_2 \subseteq A$ and $B_2 \subseteq B$ retain their original vertex labels, the lexicographic edge order of $K_{|A_2|, |B_2|}^{\prec}$ inherits from $K_{m+s-1, m+t-1}^{\prec}$. As $|A_1| \leq m-1$ and $|B_1| \leq m-1$, we have 
$$|A_2| = |A| - |A_1| = (m+s-1) - |A_1| \geq s$$ 
and 
$$|B_2| = |B| - |B_1| = (m+t-1) - |B_1| \geq t.$$
Thus, we obtain a blue edge-ordered $K_{s,t}^{\prec}$ with partite sets $A_2$ and $B_2$, a contradiction. Therefore, $\hat{r}_{\text{edge}}(mK_2^{\prec}, K_{s,t}^{\prec}) \leq m^2 + m(s + t - 2) + (s-1)(t-1)$.

We now consider the lower bound. If $\max \left\{m, st\right\} \geq cs - \binom{m-1}{2}$, then $\hat{r}_{\text{edge}}(mK_2^{\prec}, K_{s,t}^{\prec}) \geq \max \left\{m, st\right\}$ by (\ref{e3}). Suppose that $cs - \binom{m-1}{2} \geq \max \left\{m, st\right\}$. Let
$$
\mathcal{R}_{\text{edge}}(mK_2^{\prec}, K_{s,t}^{\prec}) = \{ G^{\prec} : G^{\prec} \to (mK_2^{\prec}, K_{s,t}^{\prec}) \text{ but } G'^{\prec} \nrightarrow (mK_2^{\prec}, K_{s,t}^{\prec}) \text{ for any } G'^{\prec} \varsubsetneq G^{\prec} \}.
$$
For the lower bound, let $G^{\prec} \in \mathcal{R}_{\text{edge}}(mK_2^{\prec}, K_{s,t}^{\prec})$ (a minimal edge-ordered Ramsey subgraph). It suffices to show that $|E(G^{\prec})| \geq cs - \binom{m-1}{2}$.

\begin{fact}\label{f1}
$\delta(G^{\prec}) \geq s$ and $G^{\prec}$ must contain $K_{s,t}^{\prec}$ as an edge-ordered subgraph.
\end{fact}

\begin{proof}
 Assume, for contradiction, that there exists a vertex $v \in V(G^{\prec})$ such that $d_{G^{\prec}}(v) \leq s-1$. Color all edges adjacent to $v$ blue. Since $G^{\prec} \to (mK_2^{\prec}, K_{s,t}^{\prec})$, it follows that $G^{\prec} - v \to (mK_2^{\prec}, K_{s,t}^{\prec})$, which contradicts the minimality of $G^{\prec}$. Moreover, $G^{\prec}$ must contain $K_{s,t}^{\prec}$ as an edge-ordered subgraph. Otherwise, coloring all edges of $G^{\prec}$ blue would contradict $G^{\prec} \to (mK_2^{\prec}, K_{s,t}^{\prec})$.   
\end{proof}

By Theorem \ref{FSS80} and (\ref{e-1}), we have $r_{\text{edge}}(mK_2^{\prec}, K_{s,t}^{\prec}) \geq r(mK_2, K_{s,t})=\max \{m + s + t - 1, 2m + s - 1\}$.
For $t \geq m$, we have $|V(G^{\prec})| \geq m + s + t - 1$. By Fact \ref{f1}, we have 
$$
|E(G^{\prec})| \geq st + s(m - 1) - \binom{m-1}{2} = s(t + m - 1) - \binom{m-1}{2}.
$$
For $m \geq t$, we have $|V(G^{\prec})| \geq 2m + s - 1$. Similarly, by Fact \ref{f1}, we have 
$$
|E(G^{\prec})| \geq st + s(2m - t - 1) - \binom{m-1}{2} = s(2m - 1) - \binom{m-1}{2}.
$$

Note that $\hat{r}(mK_2, K_{1,t}) = mt$ and $\hat{r}(mK_2, C_4) = 4m$ in \cite{EF81}. By Proposition \ref{prop-1}, we have $\hat{r}_{\text{edge}}(mK_2^{\prec}, K_{1,t}^{\prec}) = mt$ and $\hat{r}_{\text{edge}}(mK_2^{\prec}, K_{2,2}^{\prec}) = 4m$. This implies that the upper bound of Theorem \ref{th-3-5} is sharp. In particular, $\hat{r}_{\text{edge}}(mK_2^{\prec}, K_{1,2}^{\prec}) = 2m$. Thus, the bounds established in Theorem \ref{th-3-5} are sharp.
\end{proof}

\subsection{Complete bipartite graphs}
\indent
\par
Let $K_{1,n}$ be a star graph with center vertex $v$ and leaf vertices $v_1, v_2, \ldots, v_n$. We define the edge-ordered star graph $K_{1,n}^{\prec}$ by equipping $K_{1,n}$ with the edge order. This total order on the edge set $E(K_{1,n})$ is specified such that $(v, v_1) \prec (v, v_2) \prec \cdots \prec (v, v_n).$

The neighborhood $N_{G^{\prec}}(v)$ of a vertex $v \in V(G)$ is contained in $V(G^{\prec})$, and $d_{G^{\prec}}(v) = |N_{G^{\prec}}(v)|$. Similarly, $N^R_{G^{\prec}}(v)$ and $N^B_{G^{\prec}}(v)$ are the red and blue neighborhoods of $v \in V(G^{\prec})$, respectively, and $d_{G^{\prec}}^R(v)=|N_{G^{\prec}}^R(v)|$ and $d_{G^{\prec}}^B(v)=|N_{G^{\prec}}^B(v)|$.

\begin{theorem}\label{th-4-1}
For any positive integers $n,s,t$, we have $$\hat{r}_{\text{edge}}(K_{1,n}^{\prec}, K_{s,t}^{\prec}) \leq s^2 (n-1)+st.
$$
Moreover, the upper bound is sharp.  
\end{theorem}

\begin{proof}
Let $F^{\prec}$ be an edge-ordered complete bipartite graph with partition classes $A$ and $B$, where $|A| = s$ and $|B| = s(n-1)+t$, and let the vertices in $A$ be denoted by $\{v_1, \ldots, v_s\}$.  Assume that there is no red edge-ordered  $K_{1,n}^{\prec}$ for any red/blue edge-coloring of $F^{\prec}$, we will show that there is a blue edge-ordered $K_{s,t}^{\prec}$. Since there is no red edge-ordered  $K_{1,n}^{\prec}$, it follows that $d_{F^{\prec}}^R(v_i) \leq n-1$ for each $1 \leq i \leq s$.
Then $d_{F^{\prec}}^B(v_i) \geq |B|-(n-1)$ for each $1 \leq i \leq s$. Furthermore, we have 
$$
|N_{F^{\prec}}^B(v_1) \cap N_{F^{\prec}}^B(v_2) \cap \ldots \cap N_{F^{\prec}}^B(v_s)| \geq |B|-s(n-1)=t.
$$
These common neighbors form a blue edge-ordered $K_{s,t}^{\prec}$ with $\{v_1, \ldots, v_s\}$. Since $K_{s,t}^{\prec}$ is a subgraph of $F^{\prec}$, its edge order is necessarily preserved. This implies that we obtain the desired blue edge-ordered copy of $K_{s,t}^{\prec}$.

For positive integers $n$ and $t$, we have
$\hat{r}(K_{1,n}, K_{1,t})=n+t-1$, a result established in \cite{EFRS78}. By Theorem \ref{th-4-1} and (\ref{e1}), we have $\hat{r}_{\text{edge}}(K_{1,n}^{\prec}, K_{1,t}^{\prec})=n+t-1$. Faudree et al. \cite{FRS84} proved that $\hat{r}(K_{1,n}, K_{2,t})=4n+2t-4$ for $t \geq 9$ and $n$ is sufficiently large, and also showed that $\hat{r}(K_{1,n}, K_{2,2})=4n$ for $n \geq 3$.
Similarly, by Theorem \ref{th-4-1} and (\ref{e1}), we have $\hat{r}_{\text{edge}}(K_{1,n}^{\prec}, K_{2,t}^{\prec})=4n+2t-4$ for $t \geq 9$ and $n$ is sufficiently large, and $ \hat{r}_{\text{edge}}(K_{1,n}^{\prec}, K_{2,2}^{\prec})=4 n $ for $n \geq 3$.
\end{proof}

\begin{theorem}\label{th-4-2}
For any integers $t \geq s \geq 2$, we have $\hat{r}_{\text{edge}}(K_{s,t}^{\prec}) \leq e s^2 2^{s+2} t$.
\end{theorem}

\begin{proof}
Let $F^{\prec}$ be a edge-ordered complete bipartite graph with partition classes $A$ and $B$, where $|A| = 2s^2$ and $|B| = e 2^{s+1} t$.  Then $|E(F^{\prec})|= e s^2 2^{s+2} t$. It suffices to show that any red/blue coloring of $F^{\prec}$ contains either a red edge-ordered $K_{s,t}^{\prec}$ or a blue edge-ordered $K_{s,t}^{\prec}$.
We begin by splitting the vertex set $B$ into two disjoint subsets. A vertex $v \in B$ is designated \textit{red} whenever at least half of its incident edges are red, and designated \textit{blue} in all other instances.

Since $|B| = e 2^{s+1} t$, by the pigeonhole principle, at least one class (red or blue) contains at least $e2^s t$ vertices. Without loss of generality, assume that there are at least $e 2^s t$ red vertices, denote these vertices by $B_1$. Each vertex $v \in B_1$ has degree $2s^2$ in $A$. By the definition of a red vertex, at least $s^2$ of the edges incident to $v$ are red. It follows that every vertex $v$ gives rise to at least $\binom{s^2}{s}$ distinct red stars $K_{1,s}^{\prec}$. The total number of red stars is at least 
$$
|B_1|\cdot \binom{s^2}{s} \geq e 2^s t \binom{s^2}{s}.
$$

There are $\binom{2s^2}{s}$ subsets of size $s$ (leaf sets) in $A$.
Each red star corresponds to a leaf set. Thus, for $s \geq 2$, the average number of occurrences per leaf set is at least
$$
\begin{aligned}
\frac{e 2^s t\binom{s^2}{s}}{\binom{2s^2}{s}} 
&=e 2^s t \prod_{i=0}^{s-1} \frac{s^2 - i}{2s^2 - i} \\
&\geq e 2^s t \prod_{i=0}^{s-1} \frac{1}{2}\left(1 - \frac{i}{s^2}\right) 
= e t \prod_{i=1}^{s-1} \left(1 - \frac{i}{s^2}\right)\\
& \geq e t\prod_{i=1}^{s-1} e^{-2i/s^2} = et \exp\left(\frac{-2}{s^2}\sum_{i=1}^{s-1} i \right)=et \exp\left(-\frac{s-1}{s} \right)\\
& \geq e t e^{-1}=t,
\end{aligned}
$$
where the second inequality use $1-x \geq e^{-2x}$ for $0 \leq x \leq \frac{1}{2}$, since $0 \leq \frac{i}{s^2} \leq \frac{s-1}{s^2} < \frac{1}{2}$ for $1 \leq i \leq s-1$ and $s \geq 2$. Therefore, there exists a leaf set (an $s$-subset of $A$) that appears in at least $t$ red stars. The centers of these stars (from $B$) and the leaf set form a red $K_{s,t}^{\prec}$. Since $K_{s,t}^{\prec}$ is a subgraph of $F^{\prec}$, it follows that its edge order is still preserved, that is, we have a red edge-ordered $K_{s,t}^{\prec}$, as desired.
\end{proof}

Pikhurko \cite{Pi02} proved that $\hat{r}(K_{s,t})=\Theta\bigl(s^{2} 2^{s} t\bigr)$ for fixed $s$ and sufficiently large $t$. Therefore, by Theorem \ref{th-4-2} and (\ref{e1}), we have  
the following corollary:
\begin{corollary}
$\hat{r}_{\text{edge}}(K_{s,t}^{\prec}) =\Theta\bigl(s^{2} 2^{s} t\bigr)$  for fixed $s$ and sufficiently large $t$.    
\end{corollary}

\end{document}